\documentclass[a4paper,11pt]{article}
\usepackage[backend=bibtex, 
style=alphabetic,
sortcites=true,
firstinits=true,
hyperref,
maxbibnames=99,
]{biblatex}

\bibliography{Bibliography2}{}

\usepackage[top=3.0cm, bottom=3.0cm, inner=3.0cm, outer=3.0cm, includefoot]{geometry}
\usepackage[latin1]{inputenc}
\usepackage[T1]{fontenc}
\usepackage{verbatim}

\usepackage{geometry}
\usepackage{amssymb}
\usepackage{amsmath}
\usepackage{graphicx}
\usepackage{amsthm}

\usepackage{color}

\usepackage{enumerate}

\newcommand{\Hm}[1]{\leavevmode{\marginpar{\tiny%
$\hbox to 0mm{\hspace*{-0.5mm}$\leftarrow$\hss}%
\vcenter{\vrule depth 0.1mm height 0.1mm width \the\marginparwidth}%
\hbox to 0mm{\hss$\rightarrow$\hspace*{-0.5mm}}$\\\relax\raggedright
#1}}}
\newcommand{\eat}[1]{}

\newcommand{\IR}{{\mathbb{R}}}

\newcommand{\N}{{\mathbb{N}}}

\newcommand{\eps}{{\varepsilon}}

\newcommand{\eChar}{\begin{enumerate}[(i)]}
\newcommand{\eBr}{\begin{enumerate}[(1)]}

\newcommand{\LL}{\Delta}

\newcommand{\supp}{\operatorname{supp}}

\newcommand{\co}{\colon}
\newcommand{\D}{\mathcal{D}}
\newcommand{\Q}{\mathcal{Q}}
\newcommand{\F}{\mathcal{F}}

\newcommand{\ov}[1]{\overline{ #1}}

\title
{
Neumann semigroup, subgraph convergence, form uniqueness, stochastic completeness
and the Feller property
}

\author{Matthias Keller, Florentin M{\"u}nch, Rados{\l}aw K. Wojciechowski}
\date{\today}

\theoremstyle{plain}
\newtheorem{lemma}{Lemma}[section]
\newtheorem{theorem}[lemma]{Theorem}

\newtheorem{corollary}[lemma]{Corollary}

\theoremstyle{definition}
\newtheorem{definition}[lemma]{Definition}

\newtheorem{example}[lemma]{Example}

\newtheorem{rem}[lemma]{Remark}

\newcommand{\Deg}{\operatorname{Deg}}

\numberwithin{equation}{section}

\begin{document}



\maketitle

\begin{abstract}
We study heat kernel convergence of induced subgraphs with Neumann boundary conditions.
We first establish convergence of the resulting semigroups to the Neumann semigroup in $\ell^2$. 
While convergence to the Neumann semigroup always holds, convergence to the 
Dirichlet semigroup in $\ell^2$ turns out to be equivalent to the coincidence of the Dirichlet and Neumann semigroups
while convergence in $\ell^1$ is equivalent to stochastic completeness.
We then investigate the Feller property for the Neumann semigroup via generalized solutions and
give applications to graphs satisfying a condition on the edges as well as 
birth-death chains.
\end{abstract}




\pagestyle{plain}



\section{Introduction}
It is well-known that when exhausting a space and taking heat semigroups with Dirichlet boundary
conditions, the resulting semigroups converge to the minimal semigroup on the entire space. For manifolds,
this was first established by Dodziuk in 1983 \cite{Dod83}; for graphs as studied
in this article by Keller/Lenz in 2012 \cite{KL12}. A main ingredient in proving
this convergence is the fact that both the heat semigroups and the resolvents increase as the domains
of the exhaustion increase as can be established by a minimum principle. 
This monotonicity yields a number of consequence, for example, that the Dirichlet semigroup
generates the minimal positive solution to the heat equation and the Dirichlet resolvent generates
the minimal positive solution to the Poisson equation.

When considering analogous questions for Neumann boundary conditions, this monotonicity
no longer holds. However, we can still study the question of what happens with convergence as
one exhausts the space and takes the limit. This is the aim of this article.

More specifically, we first establish by general theory
that the restricted Neumann semigroups and resolvents converge strongly
in $\ell^2$ to the Neumann semigroup and resolvent. As a consequence, we show that the
Neumann restrictions converge to the Dirichlet semigroup strongly in $\ell^2$
if and only if the Neumann and Dirichlet 
semigroups coincide.
As the semigroups and resolvents extend to all $\ell^p$ spaces, we can then consider
other notions of convergence.
In particular, we show that convergence
to the Dirichlet semigroup in $\ell^1$ coincides with stochastic completeness.

Next, we consider the Feller property for the Neumann semigroup. This property deals with the
phenomenon of heat vanishing at infinity. 
We show that the Neumann semigroup is Feller if and only if the Dirichlet
semigroup is Feller and the semigroups coincide. We then further analyze this property 
via $\alpha$-superharmonic functions and give additional connections to stochastic completeness.

Finally, we apply our results to two special classes of graphs. The first class is graphs that satisfy
an edge condition as introduced in \cite{KM}. Here, we establish that, when such graphs 
are locally finite, the Neumann semigroup is Feller.
In the final section, we characterize the Feller property of the Neumann semigroup for birth-death chains
and provide a connection to essential self-adjointness.

We note that many of our methods are robust enough to hold in a much more general setting,
including that of manifolds.

\section{The Setting}
We study graphs as in \cite{KL12, KLW21}. More specifically, we let $X$ be a countable
discrete set of vertices with a measure $m \co X \longrightarrow (0,\infty)$ of full support.
We let the edge structure be given by a map
 $b \co X \times X \longrightarrow [0,\infty)$ which satisfies
 $b(x,y)=b(y,x)$, $b(x,x)=0$ and $\sum_{y \in X} b(x,y)<\infty$. Finally, we let
 $c \co X \longrightarrow [0,\infty)$ denote a killing term and call $(b,c)$ a \emph{graph}
 over $(X,m)$. If $c=0$, we will write just $b$ for a graph over $(X,m)$.
 We call a graph \emph{locally finite} if for every $x \in X$, there exist
 only finitely many $y \in X$ such that $b(x,y)>0$.
 
 We will assume that all graphs are connected, that is, for any pair of vertices $x,y \in X$, there
 exists a sequence of vertices $(x_k)_{k=0}^n$ such that $x_0=x, x_n=y$ and $b(x_k, x_{k+1})>0$
 for $k=0, 1, \ldots, n-1$. We call such a sequence $(x_k)$ a \emph{path} connecting $x$ and $y$.
  
 Throughout, we will consider exhaustions of the entire graph via subgraphs. More specifically, we 
 will call any sequence of finite sets $X_k \subseteq X$ an \emph{exhaustion} of the graph
 if $X_k \subseteq X_{k+1}$ and $X= \bigcup_{k=0}^\infty X_k$. We will assume that the graphs
 induced by $X_k$ are also connected, that is, if we restrict $b$ to $X_k \times X_k$, then the resulting
 graph is connected.
 
We now introduce the relevant function spaces, forms and operators. 
We let $C(X)= \{ f\co X \longrightarrow \IR\}$, $C_c(X) = \{ f\in C(X) \mid f \textup{ is finitely supported} \}$
and for $p \in [1,\infty)$,
$$\ell^p(X,m) =\{ f \in C(X) \mid \sum_{x \in X} |f(x)|^p m(x) < \infty \}$$
with norm $\| f\|_p^p = \sum_{x \in X} |f(x)|^p m(x).$ When $p=2$, we will often write $\|f\|$ for $\|f\|_2$.
We also let $\ell^\infty(X,m) = \ell^\infty(X) = \{ f\in C(X) \mid \sup_{x \in X} |f(x)| < \infty\}$ be the space
of bounded functions with norm $\|f\|_\infty = \sup_{x \in X} |f(x)|$. 

We let
$$\D = \{ f\in C(X) \mid \sum_{x,y \in X} b(x,y)(f(x)-f(y))^2 + \sum_{x \in X} c(x)f^2(x) < \infty \}$$
denote the functions of finite energy. For $f, g \in \D$, we let
$$\Q(f,g) = \frac{1}{2} \sum_{x,y \in X}b(x,y)(f(x)-f(y))(g(x)-g(y)) + \sum_{x \in X}c(x)f(x)g(x)$$
denote the energy form. 

There are two distinguished restrictions of the energy form which we will consider.
The first is the form $Q^{(D)}$ which is a restriction of $\Q$ to
$$D(Q^{(D)})= \ov{C_c(X)}^{\| \cdot \|_\Q}$$
where $\| f \|_\Q^2 = \Q(f) + \|f\|^2$. The other is $Q^{(N)}$ which is a restriction of $\Q$ to
$$D(Q^{(N)}) = \D \cap \ell^2(X,m).$$
We note that both $Q^{(D)}$ and $Q^{(N)}$ are Dirichlet forms
and, thus, the associated semigroups are positivity preserving. In fact,
when the underlying graph is connected, the semigroups are even positivity
improving, see Lemma~5.1 in Section~5.1 of \cite{KLW21}.
If $Q^{(D)}=Q^{(N)}$, then we say that the graph satisfies \emph{form uniqueness}, 
see \cite{HKLW12, KLW21} for more details.

We now introduce the corresponding formal operator and restrictions. 
We let $\F=\{ f\in C(X) \mid \sum_{y \in X}b(x,y)|f(y)| < \infty \textup{ for all } x \in X\}$
and for $f \in \F$ and $x \in X$, we let
$$\LL f(x) = \frac{1}{m(x)} \sum_{y \in X} b(x,y)(f(x)-f(y)) + \frac{c(x)}{m(x)} f(x)$$
denote the formal Laplacian. It follows from a Green's formula that the operators $L^{(D)}$
and $L^{(N)}$ coming from the forms $Q^{(D)}$ and $Q^{(N)}$ are restrictions of $\LL$ see,
e.g., Theorem~1.12 in Section~1.1 of \cite{KLW21}. 
We refer to $L^{(D)}$ as the \emph{Dirichlet Laplacian} and $L^{(N)}$
as the \emph{Neumann Laplacian}. 

Generalizing the preceeding, we say that a form $Q$ with domain $D(Q) \subseteq \ell^2(X,m)$ is \emph{associated
to the graph} if $Q$ is closed, 
$$D(Q^{(D)}) \subseteq D(Q) \subseteq D(Q^{(N)})$$ 
and if $Q(f) = \Q(f)$
for all $f \in D(Q)$. The operator that comes from such a form $Q$ is said to be \emph{associated
to the graph}. As for the Dirichlet and Neumann Laplacians,
it can be shown that all such operators are restrictions of $\LL$.

We now take an exhaustion sequence $(X_k)$ and consider two possible resulting forms
and Laplacians. 
For each $k$, we let $i_k \colon C(X_k) \longrightarrow C(X)$ denote extension by $0$. 
Then, we let the form $Q^{(D)}_k$
be defined on $\ell^2(X_k,m)$ by
$$Q^{(D)}_k(f) = \Q (i_k \circ f).$$
If we let $L^{(D)}_k$ denote the operator associated to $Q^{(D)}_k$, then a calculation gives $L^{(D)}_k f= \LL (i_k \circ f)$ 
and thus
$$L^{(D)}_k f(x) = \frac{1}{m(x)} \sum_{y \in X_k} b(x,y)(f(x)-f(y)) + 
\frac{1}{m(x)}\left(\sum_{y \not \in X_k}b(x,y) +c(x)\right)f(x)$$
for $x \in X_k$ and $f \in C(X_k)$.
We call the operator $L^{(D)}_k$ the \emph{Dirichlet restriction of the Laplacian} and the associated
semigroup $P_{t,k}^{(D)}=e^{-tL^{(D)}_k}$ the \emph{Dirichlet restriction of the semigroup}. We remark that the Dirichlet restriction
sees both edges within the exhaustion set and those that leave the set. 

As a second possible approximation, we now look at the form and Laplacian on the graph induced by the exhausting
sets $X_k$.
That is, we restrict $b$ to $X_k \times X_k$, restrict $c$ and $m$ to $X_k$ and consider the form and Laplacian
of the resulting graph. More specifically, we consider the form
$$Q^{(N)}_k(f) = \frac{1}{2} \sum_{x,y \in X_k}b(x,y)(f(x)-f(y))^2 + \sum_{x \in X_k}c(x)f^2(x)$$
and operator
$$L^{(N)}_k f(x) = \frac{1}{m(x)} \sum_{y \in X_k} b(x,y)(f(x)-f(y)) + \frac{c(x)}{m(x)}f(x)$$
for $x \in X_k$ and $f \in \ell^2(X_k, m)$. We call $L^{(N)}_k$ the \emph{Neumann restriction of the Laplacian}
and the semigroup $P_{t,k}^{(N)}=e^{-tL^{(N)}_k}$ the \emph{Neumann restriction of the semigroup}. 
We note that by restricting and extending by 0, we can think of both the forms $Q_k^{(D)}$
and $Q_k^{(N)}$ and semigroups $P_{t,k}^{(D)}$
and $P_{t,k}^{(N)}$ as being defined on all of $\ell^2(X,m)$.

We obtain from the formulas above and the Lie-Trotter product formula that
$$P_{t,k}^{(D)} \varphi \leq P_{t,k}^{(N)}\varphi$$
for all $\varphi \in \ell^2(X_k,m)$ with $\varphi \geq 0$.
We also remark that, in terms of the formal Laplacian on the entire graph, we can write
$$L^{(N)}_k f(x)= \LL (i_k \circ f)(x) - f(x)\LL 1_{X_k}(x)$$
for $x \in X_k$ and $f \in \ell^2(X_k,m)$, where $1_{X_k}$ is the indicator function of the set $X_k$.

\section{Approximation of the Neumann semigroup and resolvent} 
As mentioned in the introduction, the Dirichlet restrictions of the semigroup
converge strongly in $\ell^2(X,m)$ to the semigroup $P_{t}^{(D)}=e^{-tL^{(D)}}$ on the entire space.
The main tool in establishing this is \emph{domain monotonicity} which states that
$$P_{t,k}^{(D)} \varphi \leq P_{t,k+1}^{(D)} \varphi$$
for all $\varphi \in \ell^2(X_k,m)$ with $\varphi \geq 0$, see \cite{KL12, KLW21}
for details on this in the graph setting.

Domain monotonicity no longer holds in general for the Neumann restrictions.
However, we can still recover strong convergence for the Neumann restrictions as we will
see in this section. In order to prove this, we work with the Neumann resolvents and show 
that they converge strongly. By general theory linking resolvents and semigroups, it follows
that the semigroups converge strongly as well.

We note that the following result is a special case of general arguments found, e.g., as
Theorem~3.13a in Section~3, Chapter~8 of \cite{Kat95}. 
We give a proof here for the sake of completeness and for the convenience of the reader.

\begin{theorem}[Strong Neumann convergence]\label{thm:Neumann}
Let $(b,c)$ be a graph over $(X,m)$.
Let $L^{(N)}$ be the Neumann Laplacian with resolvent $R_\alpha^{(N)} = (L^{(N)}+\alpha)^{-1}$ for $\alpha>0$
and semigroup $P_{t}^{(N)}=e^{-tL^{(N)}}$ for $t \geq 0$.
Let $R^{(N)}_{\alpha,k} = (L^{(N)}_k+\alpha)^{-1}$ for $\alpha>0$
and $P_{t,k}^{(N)}=e^{-tL_k^{(N)}}$ for $t \geq 0$ be the restricted Neumann resolvents
and semigroups for an exhaustion $(X_k)$. Then,
\begin{itemize}
\item[\textup{(i)}] $R^{(N)}_{\alpha,k} \to R^{(N)}_\alpha$ strongly 
in $\ell^2(X,m)$ as $k \to \infty$ for every $\alpha>0$.
\item[\textup{(ii)}] $P_{t,k}^{(N)} \to P_{t}^{(N)}$ strongly
in $\ell^2(X,m)$ as $k \to \infty$ for every $t \geq 0$.
\end{itemize}

\end{theorem} 
\begin{proof}
By general principles connecting semigroups and resolvents, (i) and (ii) are equivalent. Thus, it suffices
to prove (i).

Let $f \in \ell^2(X,m)$.
We prove (i) via the following claims:
\begin{enumerate}[(a)]
\item The resolvents $R_{\alpha,k}^{(N)} f$ converge pointwise to some $R_{\alpha}f$.
\item This limit coincides with $R_\alpha^{(N)} f$.
\item The resolvents $R_{\alpha,k}^{(N)} f$ also converge in $\ell^2(X,m)$.
\end{enumerate}

We start by proving claim (a).
Let $Q_k^{(N)}$ be the Neumann form corresponding to $L^{(N)}_k$ which
we think of as being defined on $\ell^2(X,m)$
and let $R_{\alpha,k}^{(N)}$ be the corresponding resolvent.

By general theory, $R_{\alpha,k}^{(N)} f$ is the unique minimizer of the functional
\[
\psi_{f,\alpha,k}(v)= Q_k^{(N)}(v) + \alpha \|v - \frac{1}{\alpha} f\|^2
\]
for $v \in D(Q^{(N)})$ see, e.g., Theorem~E.1 in \cite{KLW21}.
Similarly, $R_\alpha^{(N)} f$ is the unique minimizer of the functional
\[
\psi_{f,\alpha}(v)= Q^{(N)}(v) + \alpha \|v - \frac 1 \alpha f\|^2
\]
for $v \in D(Q^{(N)})$.

In the following, let $u = R_{\alpha,k}^{(N)}f= (L^{(N)}_k +\alpha)^{-1} f$ be the minimizer of 
$\psi_{f,\alpha,k}$. We remark that $L_k^{(N)} u = f- \alpha u$
and use this to calculate
\begin{align*}
Q_k^{(N)}(u) + \alpha \|u - \frac 1 \alpha f\|^2  &=  \langle L_k^{(N)} u,u \rangle
 +  \alpha \langle u -  \frac 1 \alpha f, u -\frac 1 \alpha f \rangle \\
&= \langle f - \alpha u ,u \rangle - \langle f - \alpha u, u -\frac 1 \alpha f \rangle\\
&=\langle f- \alpha u, \frac 1 \alpha f \rangle = \frac 1 \alpha \|f\|^2 - \langle u ,f \rangle \\
&= \frac 1 \alpha \|f\|^2 - \langle R_{\alpha,k}^{(N)} f ,f \rangle.
\end{align*}
Since $Q_k^{(N)}$ is increasing in $k$, 
it follows that $\langle R_{\alpha,k}^{(N)} f ,f \rangle$ is decreasing in $k$. As $R_{\alpha,k}^{(N)}$ is a positive operator, i.e., $\langle R_{\alpha,k}^{(N)} f ,f \rangle \geq 0$, it 
follows that $\langle R_{\alpha,k}^{(N)} f ,f \rangle$ converges to 
some $\langle R_\alpha f,f \rangle$.
By polarization, this yields the pointwise convergence of $R_{\alpha,k}^{(N)} f$ to $R_\alpha f$.
This finishes the proof of claim (a).

We next prove claim (b).
By Fatou's lemma, and using $L_k^{(N)} u = f- \alpha u$ for $u = R_{\alpha,k}^{(N)}f$ as
well as $\|R_{\alpha,k}^{(N)}f\| \leq \|f\|/\alpha$, we have
\[
Q^{(N)}(R_\alpha f) \leq \liminf_{k \to \infty} Q_k^{(N)}(R_{\alpha,k}^{(N)} f) = \liminf_{k \to \infty}
\langle f-\alpha u, R_{\alpha,k}^{(N)} f\rangle < \infty
\]
and
\[
\|R_\alpha f -\frac 1 \alpha f\|^2 \leq  \liminf_{k \to \infty} \|R_{\alpha,k}^{(N)} f - \frac 1 \alpha f\|^2 < \infty.
\]
In particular, $R_\alpha f \in D(Q^{(N)})$ and, letting $\phi_{\alpha,k}(f)$ denote the minimal value of the functional $\psi_{f,\alpha,k}$
and $\phi_{\alpha}(f)$ denote the minimal value of the functional $\psi_{f,\alpha}$, we obtain
\[
\psi_{f,\alpha}(R_\alpha f) \leq \liminf_{k \to \infty} \psi_{f,\alpha,k} (R_{\alpha,k}^{(N)}f) = \liminf_{k \to \infty} \phi_{\alpha,k}(f) \leq \phi_{\alpha}(f)
\]
where the latter inequality follows from $Q_k^{(N)} \leq Q^{(N)}$.
Therefore, $R_\alpha f$ is the unique minimizer of $\psi_{f,\alpha}$
and thus $R_\alpha f = R_\alpha^{(N)} f$ which proves claim (b).
In particular, $\langle R_{\alpha,k}^{(N)}f,f \rangle$ converges from above to $\langle R_\alpha^{(N)} f,f \rangle$ as 
$k \to \infty$.

We now prove claim (c), i.e., that $R_{\alpha,k}^{(N)}f$ converges to $R_\alpha^{(N)} f$ in $\ell^2(X,m)$ as $k \to \infty$.
Let $T_{\alpha,k}= R_{\alpha,k}^{(N)} - R_\alpha^{(N)}$. Then, $T_{\alpha,k} \geq 0$ is a bounded self-adjoint operator 
such that  $\|T_{\alpha,k}^{1/2}f\|^2 = \langle T_{\alpha,k} f, f \rangle \to 0$ as $k \to \infty$.
Thus,
\[
\|T_{\alpha, k} f\| \leq \|T_{\alpha,k}^{1/2}\| \cdot \|T_{\alpha,k}^{1/2} f\| \to 0
\]
as $k \to \infty$ since $\|T_{\alpha,k}^{1/2}\|$ is uniformly bounded in $k$ as $\|R_{\alpha,k}^{(N)}\| \leq 1/\alpha$.
Therefore,
$R_{\alpha,k}^{(N)}f$ converges to $R_\alpha^{(N)} f$ in $\ell^2(X, m)$ as $k \to \infty$ which proves claim (c)
and the theorem.
\end{proof}

\begin{rem}
We note that, in contrast to the Dirichlet semigroup, the Neumann semigroup will not, in general,
generate the largest positive solution of the heat equation and the Neumann resolvent will not, in general,
generate the largest positive solution to the Poisson equation. This can be inferred from characterizations
of all Dirichlet forms between the Dirichlet and Neumann forms via boundary representations, see \cite{KLSS19},
as well as domination of semigroup results presented in \cite{LSW21}.
\end{rem}

Having established that the restricted Neumann semigroups always converge strongly to the Neumann
semigroup in $\ell^2$, we now look at convergence to the Dirichlet semigroup. 

\begin{corollary}[Characterization of form uniqueness]\label{cor:form_uniqueness}
Let $(b,c)$ be a connected graph over $(X,m)$.
Let $t>0$. The following statements are equivalent:
\begin{itemize}
\item[\textup{(i)}] $Q^{(D)}=Q^{(N)}$.
\item[\textup{(ii)}] $P_{t,k}^{(N)} \to P_{t}^{(D)}$ strongly in $\ell^2(X,m)$
as $k \to \infty$.
\item[\textup{(iii)}] $P_{t,k}^{(N)} \varphi \to P_{t}^{(D)} \varphi$ pointwise 
as $k \to \infty$ for some (all) $\varphi \in C_c(X)$ such that $\varphi \not = 0$. 

\end{itemize}
\end{corollary}

\begin{proof}
The implication (i) $\Longrightarrow$ (ii) follows from Theorem~\ref{thm:Neumann}
and (ii) $\Longrightarrow$ (iii) is trivial.

We now show  (iii) $\Longrightarrow$ (i).
Without loss of generality we assume $\varphi \geq 0$. Suppose that $Q^{(D)} \neq Q^{(N)}$. 
Then, there exist vertices $x, y \in X$
such that $P_{t}^{(D)}1_x(y) < P_{t}^{(N)}1_x(y)$ for some $t>0$. 
By connectedness of the graph and positivity improving 
properties of the semigroup, it now follows that $P_{t}^{(D)}1_x < P_{t}^{(N)}1_x$ for all $x \in X$, e.g., compare the proof of Theorem~1.26
in Section~1.4 of \cite{KLW21}. By the positivity of $\varphi$
and symmetry of the semigroup, the conclusion now follows.
\end{proof}

\begin{rem}
We note that form uniqueness is also equivalent to Markov uniqueness, i.e., that the Laplacian
has a unique realization whose form is a Dirichlet form, see \cite{Sch20} or Theorem~3.12 in Section~3.3 of
\cite{KLW21}. 
\end{rem}

\section{Stochastic Completeness}
Having established that the convergence of the restricted Neumann semigroups to the Dirichlet
semigroup pointwise is equivalent to form uniqueness, we now consider another form of convergence,
namely, strong convergence in $\ell^1$. 
We will see that this form of convergence is equivalent
to stochastic completeness.

For this, we let $c=0$ and recall that the Dirichlet heat semigroup extends to all $\ell^p(X,m)$ spaces
including $\ell^\infty(X)$, see Section~2.1 in \cite{KLW21} for details. 
We say that a graph is \emph{stochastically complete} if
$$P_{t}^{(D)}1=1$$
for some (all) $t >0$ where $1$ denotes the constant function which is $1$ on every vertex.

We make some preliminary remarks before proving our characterization of stochastic completeness
in terms of exhaustions. First, we note that stochastic completeness is equivalent to the preservation of
$\ell^1$ norms, i.e.,
$$\|P_{t}^{(D)} \varphi\|_1 = \|\varphi\|_1$$
for some (all) $\varphi \in C_c(X)$ with $\varphi \geq 0$, $\varphi \neq 0$ and for some (all) $t \geq 0$.
Secondly, we note that the Neumann restricted Laplacians are graph Laplacians for finite
graphs that have no killing term as we have assumed $c=0$. Thus, 
the associated semigroups are stochastically complete, i.e.,
$$\|P_{t,k}^{(N)}\varphi\|_1 = \|\varphi\|_1$$
for every $\varphi \in C_c(X)$ with $\varphi \geq 0$ and every $t \geq 0$ see, e.g., Theorem~0.65 in Section~0.8 of \cite{KLW21}.
Finally, as noted previously, for $\varphi \geq 0$, we have
$$P_{t,k}^{(D)} \varphi \leq P_{t,k}^{(N)} \varphi$$
and
$$P_{t,k}^{(D)}\varphi \leq P_{t}^{(D)}\varphi$$
for all $t \geq 0$ and all $k$.

Using these observations, we now characterize stochastic completeness.

\begin{theorem}[Characterization of stochastic completeness]\label{thm:SC}
Let $b$ be a connected graph over $(X,m)$.
Let $t > 0$.
The following statements are equivalent:
\begin{itemize}
\item[\textup{(i)}]
The graph $b$ over $(X,m)$ is stochastically complete.
\item[\textup{(ii)}]
$P_{t,k}^{(N)} \to P_{t}^{(D)}$ strongly in $\ell^1(X,m)$ as $k \to \infty$.
\end{itemize}

\end{theorem}

\begin{proof}
We let $\varphi \in C_c(X)$ with $\varphi \geq 0$ and $ \varphi \neq 0$.

We first prove (i) $\Longrightarrow$ (ii). We use the triangle inequality along with the properties
above to estimate
\begin{align*}
\| P_{t,k}^{(N)} \varphi -  P_{t}^{(D)} \varphi  \|_1  &\leq \| P_{t,k}^{(N)} \varphi  -  P_{t,k}^{(D)} \varphi  \|_1 
+ \| P_{t}^{(D)} \varphi  -  P_{t,k}^{(D)}\varphi  \|_1 \\
&=  \| P_{t,k}^{(N)} \varphi  \|_1 - \|  P_{t,k}^{(D)} \varphi  \|_1 +  \|  P_{t}^{(D)} \varphi  \|_1 - \| P_{t,k}^{(D)}  \varphi \|_1 \\
&= 2(\|\varphi\|_1 -  \|  P_{t,k}^{(D)} \varphi  \|_1) 
\end{align*}
where the first equality follows from  $\varphi \geq 0$ and from the facts discussed
directly above the theorem and the second equality follows from stochastic completeness.
Moreover, due to stochastic completeness we have $\|P_{t}^{(D)} \varphi\|_1 = \|\varphi\|_1$ and since $\varphi \geq 0$ 
$$ \|\varphi\|_1 -  \|  P_{t,k}^{(D)} \varphi  \|_1 =
\| P_{t}^{(D)}\varphi\|_1 -  \|  P_{t,k}^{(D)} \varphi  \|_1=\| P_{t}^{(D)}\varphi- P_{t,k}^{(D)} \varphi  \|_1.$$ 
This can be made arbitrarily small when choosing $k$ large enough due to monotone convergence. 
This proves (i) $\Longrightarrow$ (ii).

We now prove (ii) $\Longrightarrow$ (i). Suppose that $b$ over $(X,m)$ is stochastically incomplete.
We observe
\[
\| P_{t,k}^{(N)}\varphi  -  P_{t}^{(D)} \varphi  \|_1 \geq \|P_{t,k}^{(N)} \varphi  \|_1 - \|  P_{t}^{(D)} \varphi  \|_1 
= \|\varphi\|_1 -  \|  P_{t}^{(D)} \varphi  \|_1 >0
\]
independently of $k$ due to the stochastic incompleteness and connectedness of the graph. 
Therefore, $P_{t,k}^{(N)} \varphi $ cannot converge to  $P_{t}^{(D)} \varphi$ in $\ell^1(X,m)$ which contradicts (ii).
This finishes the proof.
\end{proof}

We present a corollary of the preceding results for graphs of finite measure, 
i.e., when $m(X)<\infty$, see \cite{GHKLW15}.
In this case, convergence in $\ell^2$ always implies convergence in
$\ell^1$. Our results above give that the converse holds for the Neumann
restricted semigroups converging to the Dirichlet semigroup.
\begin{corollary}[Finite measure and semigroup convergence]
Let $b$ be a connected graph over $(X,m)$ with $m(X)<\infty$. Let $t>0$.
The following statements are equivalent:
\begin{itemize}
\item[\textup{(i)}]  $P_{t,k}^{(N)} \to P_{t}^{(D)}$ strongly
in $\ell^1(X,m)$ as $k \to \infty$.
\item[\textup{(ii)}] $P_{t,k}^{(N)} \to P_{t}^{(D)}$ strongly
in $\ell^2(X,m)$ as $k \to \infty$.
\end{itemize}
\end{corollary}
\begin{proof}
When $m(X)<\infty$, stochastic completeness and form uniqueness
are equivalent, see Theorem~16 in \cite{Sch17}. Given this,
the equivalence of the two statements above follows directly
from Corollary~\ref{cor:form_uniqueness} and Theorem~\ref{thm:SC}.
\end{proof}

\begin{rem}
For some background on stochastic completeness in the case of Riemannian manifolds
see \cite{Gri99}.
Recent years have seen great interest in studying the stochastic completeness
of graphs see, e.g., \cite{DM06, Fol14b, GHM12, HL17, Hua11, Hua11a, Hua12, Hua14, HKS20, KL10, KL12, KLW13, Sch17, Web10, Woj08, Woj09, Woj11, MW19, Woj21}. Here we characterize
stochastic completeness via convergence in $\ell^1$. For further connections between stochastic completeness
and Liouville properties in $\ell^1$ see \cite{AS23}. 
\end{rem}

\section{The Feller property of the Neumann semigroup}
In this section, we study the Feller property for the Neumann semigroup. This property
has to do with heat vanishing at infinity. We will characterize it via 
a connection to form uniqueness.

More specifically, we let 
$$C_0(X)=\ov{C_c(X)}^{\| \cdot \|_\infty}$$
denote the functions vanishing at infinity.
Given a semigroup $P_{t}=e^{-tL}$ on $\ell^2(X,m)$ for an operator $L$ associated to a graph 
whose form is a Dirichlet form,
we say that the semigroup satisfies the \emph{Feller property} or is \emph{$C_0$-conservative}
if
$$P_{t}(C_c(X)) \subseteq C_0(X)$$
for some (all) $t> 0$.
We note that it suffices to consider only positive functions with finite support in the above inclusion
or even just the indicator function of a vertex provided that the graph is connected, see Lemma~3.1 in \cite{HMW19}.
Furthermore, it is equivalent to consider the vanishing of resolvents in the above
definition, i.e., that 
$$R_\alpha (C_c(X)) \subseteq C_0(X)$$
where $R_\alpha = (L +\alpha)^{-1}$ for $\alpha>0$. In particular,
if $R_\alpha 1_x \in C_0(X)$ for some vertex $x \in X$, then $P_t$ is Feller.

We now consider two semigroups involving operators associated to a graph. 
We use
the minimum principle to show
that if the bigger semigroup is Feller, then the smaller semigroup is Feller and the two semigroups, in fact, coincide.

\begin{theorem}
Let $(b,c)$ be a connected graph over $(X,m)$.
Let $L_1$ and $L_2$ be operators associated to the graph whose forms are Dirichlet forms 
with semigroups $P_{t}^{(1)}$ and $P_{t}^{(2)}$.
Suppose that $$P_{t}^{(1)}f \leq P_{t}^{(2)}f$$ for all $f \geq 0$ with $f \in \ell^2(X,m)$ and all $t\geq0$.
The following statements are equivalent:
	\begin{itemize}
		\item[\textup{(i)}] $P_{t}^{(2)}$ is Feller.
		\item[\textup{(ii)}] $P_{t}^{(1)}$ is Feller and $P_{t}^{(1)}=P_{t}^{(2)}$.
	\end{itemize}
\end{theorem}

\begin{proof} 
	The implication (ii) $\Longrightarrow$ (i) is trivial. Thus, we only have to prove (i) $\Longrightarrow$ (ii).
	Obviously, $P_{t}^{(1)}$ is Feller since $P_{t}^{(2)}$ is Feller and $P_{t}^{(1)} f \leq P_{t}^{(2)} f$ 
	for all $f \geq 0$.
	
	It is left to show $P_{t}^{(1)}=P_{t}^{(2)}$. Suppose not.
	Let $x \in X$ and let 
    $$u_t = \left(P_{t}^{(2)} - P_{t}^{(1)} \right) 1_x$$
    for $t \geq 0$.
    As both operators are restrictions of the formal Laplacian, $u_t$ satisfies the heat equation, i.e., $(\LL+\partial_t) u_t = 0$.	
    Thus, due to connectedness and positivity improving properties of the semigroups,
    we have $u_t >0$ for 
    all $t >0$.
	
	Let $T>0$ and choose $\eps>0$ so that 
	$$\eps<  e^{-T \Deg(x)} u_T(x)$$ 
	where 
	$\Deg(x) = \left( \sum_{y \in X} b(x,y)+c(x)\right)/m(x)$ 
	is the weighted degree of $x$.
	Since $P_{t}^{(2)}$ is Feller, there exists a finite set $K \subseteq X$ containing $x$ such that
	$P_{T}^{(2)} 1_x \leq \eps$ on the vertex boundary 
	$$\partial K =\{y \in K \mid b(y,z)>0 \mbox{ for some } z \in X \setminus K\}.$$
	Due to the minimum principle for the heat equation, e.g., 
	Theorem~1.10 in Section~1.1 of \cite{KLW21}, since $u_0=0$, there exist 
	$y \in \partial K$ and $t \in [0,T]$ such that 
	$$u_t(y) \geq u_T(x) >0.$$
	Furthermore, using the minimum principle again, we obtain $P_{s}^{(2)}1_x \geq e^{-s \Deg(x)}1_x$
	for all $s \geq 0$.
	Thus, using the semigroup property and putting everything together, we obtain
	\begin{align*}
	P_{T}^{(2)} 1_x(y) &= P_{t}^{(2)} P_{T-t}^{(2)} 1_x(y) \\
	& \geq P_{t}^{(2)} \left(e^{-(T-t) \Deg(x)} 1_x \right)(y) \\
	&= e^{-(T-t) \Deg(x)}  P_{t}^{(2)} 1_x(y) \\
	&\geq e^{-(T-t) \Deg(x)}u_t(y) \\
	& \geq e^{-T \Deg(x)} u_T(x) >\eps.
	\end{align*}
	This is a contradiction to 	$P_{T}^{(2)} 1_x \leq \eps$ on $\partial K$.
	Therefore, the assumption $P_{t}^{(1)}\neq P_{t}^{(2)}$ is wrong.
	This proves the theorem.
\end{proof}

We obtain an immediate corollary which characterizes the Feller property for the Neumann semigroup
via the Feller property for the Dirichlet semigroup and form uniqueness.

\begin{corollary}[Feller property and form uniqueness]\label{cor:Feller_uniqueness}
	Let $(b,c)$ be a connected graph over $(X,m)$. The following statements are equivalent:
	\begin{itemize}
	\item[\textup{(i)}] The Neumann semigroup is Feller.
	\item[\textup{(ii)}] The Dirichlet semigroup is Feller and $Q^{(D)}=Q^{(N)}$.
    \end{itemize}	
    In particular, in this case, the restricted Neumann semigroups
    converge strongly to the Dirichlet semigroup in $\ell^2(X,m)$
for any exhaustion.
\end{corollary}

\begin{rem}
For some background on the Feller property for Riemannian manifolds see \cite{Aze74, Yau78, Dod83, KL, PS12}.
For the graph case, the Feller property has been studied for the Dirichlet semigroup in 
\cite{Adr21, HMW19, Woj17}.
\end{rem}

\section{The Feller property via $\alpha$-superharmonic functions}
In this section, we investigate the relation between the Feller property of the Neumann semigroup
and the existence of $\alpha$-(super)harmonic functions for $\alpha>0$. 

More specifically, we say that a function $u \in \F$ is 
\emph{$\alpha$-superharmonic}
if $(\LL + \alpha)u \geq 0$. If $(\LL + \alpha)u = 0$, we say that $u$ is \emph{$\alpha$-harmonic}. We will shortly establish
that the triviality of positive 
$\alpha$-superharmonic functions in $\ell^1$ for $\alpha>0$ 
implies that the Neumann semigroup is Feller.

We note that there
is a characterization of the Feller property for the Dirichlet semigroup
via the vanishing at infinity of the minimal positive $\alpha$-superharmonic function, see \cite{HMW19, Woj17}.
Furthermore, form uniqueness is equivalent to the fact that all $\alpha$-harmonic functions in $D(Q^{(N)})$
vanish, see \cite{HKLW12, KLW21}.

Schematically, for locally finite graphs, we will establish the following results:

\begin{tabular}{lll}
Stochastic completeness $\&$  \quad $\Longrightarrow$ & No positive $\alpha$-harmonic \quad $\Longrightarrow$ & $P_{t}^{(N)}$ Feller. \\
\qquad \qquad $P_{t}^{(D)}$ Feller & \qquad  functions in $\ell^1$ 
\end{tabular}

Moreover, none of the implications allow for a reverse as follows from 
Examples~\ref{ex:NoSolutionStochIncomplete}~and~\ref{ex:solutionButFeller} below.
However, for the case of locally finite graphs with finite measure, all of the properties above
are equivalent by Corollary~\ref{cor:finite_measure} below.

We note that stochastic completeness always implies form uniqueness and thus the implication from 
the leftmost property to the rightmost property in the above is immediate from Corollary~\ref{cor:Feller_uniqueness}.
However, here we pass through the middle property involving $\alpha$-harmonic functions in $\ell^1$.

We start by showing the second implication in the above.

\begin{theorem}\label{thm:harmonic_Feller}
Let $(b,c)$ be a connected graph over $(X,m)$. 
If every positive $\alpha$-superhamonic function in $\ell^1(X,m)$ for some $\alpha>0$ is trivial,
then the Neumann semigroup is Feller. If the graph is locally finite, then it
suffices to consider positive $\alpha$-harmonic functions.
\end{theorem}

\begin{proof}
Suppose the Neumann semigroup is not Feller. By Corollary~\ref{cor:Feller_uniqueness}
this means either $Q^{(D)}\neq Q^{(N)}$ or $P_{t}^{(D)}$ is not Feller.
If $Q^{(D)}\neq Q^{(N)}$, there exists a non-trivial $u \geq 0$ with $(\LL+\alpha)u= 0$
and $u \in \ell^p(X,m)$ for every $p \in [1,\infty]$ see, e.g., Theorem~3.2 in Section~3.1 of \cite{KLW21}.

Now suppose that $P_{t}^{(D)}$ is not Feller. Let $R_\alpha=(L^{(D)}+\alpha)^{-1}$ denote
the resolvent associated to $L^{(D)}$ for $\alpha>0$ extended to $\ell^1(X,m)$. 
For $x \in X$, let $\delta_x = 1_x/m(x)$ denote
the indicator function of $x$ normalized so that $\|\delta_x\|_1 =1$. 
Let $\eps>0$. As $P_{t}^{(D)}$ is not Feller, there exists a sequence of 
vertices $(x_n)$ which leaves every
finite set so that $R_\alpha \delta_x (x_n) > \eps >0$ for all $n \in \N$.
By symmetry of the resolvent, we note
$$R_\alpha \delta_x(x_n) = \langle R_\alpha \delta_x, \delta_{x_n}\rangle
=\langle \delta_x, R_\alpha \delta_{x_n}\rangle = R_\alpha \delta_{x_n}(x).$$

Consider
$$u_n= \frac{R_\alpha \delta_{x_n}}{R_\alpha \delta_{x_n}(x)}.$$
Then, $u_n(x)=1$ and $\|u_n\|_1 \leq 1/(\alpha \eps)$ as $\| R_\alpha \|_1 \leq 1/\alpha$ for $\alpha>0$. By using a Harnack principle, it follows
that there exists a subsequence converging pointwise to a non-trivial $u \geq 0$ which is
$\alpha$-superharmonic in general and $\alpha$-harmonic in case the graph is locally finite,
see Corollary~4.5 in Section~4.1 of \cite{KLW21}. By Fatou's Lemma, $u \in \ell^1(X,m)$. This completes the proof.
\end{proof}

We now prove the first implication in the scheme above by showing that stochastic completeness
together with the Feller property of the Dirichlet semigroup imply that there are no positive
$\alpha$-harmonic functions in $\ell^1$. For this, we let $c=0$ and recall that stochastic
completeness is equivalent to the Dirichlet semigroup and resolvent preserving the $\ell^1$-norm
of positive functions.

\begin{theorem}\label{thm:sc+Feller}
Let $b$ be a connected graph over $(X,m)$. 
If $b$ over $(X,m)$ is stochastically complete and the Dirichlet semigroup is Feller,
then every positive $\alpha$-harmonic function in $\ell^1(X,m)$ 
for $\alpha>0$ is trivial.
\end{theorem}

\begin{proof}
Suppose there exists a non-trivial $u \geq 0$ which satisfies $(\LL + \alpha)u = 0$ for $\alpha>0$
and $u \in \ell^1(X,m)$. By connectedness and a Harnack principle, it follows that $u>0$, 
see Corollary~4.2 in Section~4.1 of \cite{KLW21}.

Let $x \in X$ and $\eps>0$. Let $R_\alpha=(L^{(D)}+\alpha)^{-1}$ denote the resolvent of $L^{(D)}$ for $\alpha>0$
on $\ell^1(X,m)$. By the Feller property of the Dirichlet semigroup, there exists $K \subseteq X$ finite with $x \in K$
such that $R_\alpha \delta_x < \eps$ on 
$$\partial K=\{ y \in K \mid \textup{there exists } z \not \in K \textup{ such that } b(y,z)>0 \}$$
where $\delta_x = 1_x/m(x)$.

Now, let
\[g= \inf\{h \in C(X) \mid (\LL+\alpha) h \geq 0, h\geq 0, h_{|_K}=u_{|_K}\}.
\]
As the resolvent generates the minimal positive solution and $g$ is 
$\alpha$-harmonic on $K \setminus \partial K$, we obtain
$g= R_\alpha v$ for some $v\geq 0$ with $\supp v \subseteq \partial K$,
see Theorem~2.12 in Section 2.2 of \cite{KLW21}.

We write
\[
v= \sum_{y \in \partial K} a_y \delta_y
\]
for suitable $a_y \geq 0$.
By minimality and stochastic completeness,
\[
\|u\|_1 \geq \|R_\alpha v\|_1 = \|v\|_1= \sum_{y \in \partial K} a_y.
\]
On the other hand, by symmetry of the resolvent,
\[
u(x)= R_\alpha v(x)=\sum_{y \in \partial K} a_y R_\alpha \delta_y(x) = 
\sum_{y \in \partial K} a_y R_\alpha \delta_x(y)\leq \eps \sum_{y \in \partial K} a_y.
\]
Thus, $\|u\|_1 \geq u(x)/\eps$ which is a contradiction since $u(x)>0$ and $\eps$ can be chosen arbitrarily small.
This finishes the proof.
\end{proof}

We now give a corollary for locally finite graphs of finite measure.

\begin{corollary}[Finite measure and the Feller property]\label{cor:finite_measure}
Let $b$ be a locally finite connected graph over $(X,m)$ with $m(X)<\infty$.
The following statements are equivalent:
\begin{itemize}
\item[\textup{(i)}]  The Neumann semigroup is Feller.
\item[\textup{(ii)}] The Dirichlet semigroup is Feller and the 
graph is stochastically complete.
\item[\textup{(iii)}]  Every positive $\alpha$-harmonic function in $\ell^1(X,m)$
for $\alpha>0$ is trivial.
\end{itemize}
In particular, if any of these additional conditions hold, then the restricted
Neumann semigroups converge strongly to the Dirichlet semigroup
in both $\ell^1(X,m)$ and $\ell^2(X,m)$ for any exhaustion.
\end{corollary}
\begin{proof}
When $m(X)<\infty$, stochastic completeness and form uniqueness
are equivalent by Theorem~16 in \cite{Sch17}. Given this,
the equivalence (i) $\Longleftrightarrow$ (ii) follows from
Corollary~\ref{cor:Feller_uniqueness}. The equivalence of these properties and (iii)
follows by combining Theorems~\ref{thm:harmonic_Feller}~and~\ref{thm:sc+Feller}. 
The ``in particular'' statement then follows from Corollary~\ref{cor:form_uniqueness} and Theorem~\ref{thm:SC}.
\end{proof}

In the following examples, we show that no reverse implications of the above theorems hold.

\begin{example}[Triviality of $\alpha$-harmonic functions does not imply stochastic
completeness]\label{ex:NoSolutionStochIncomplete}
We will shortly introduce birth-death chains and characterize the existence of positive $\alpha$-harmonic functions in $\ell^1$ for such graphs.
In particular, we will see that if such a graph has infinite measure, i.e., $m(X)=\infty$,
then there is no positive $\alpha$-harmonic function in $\ell^1(X,m)$ as all such functions must increase along the chain,
see Theorem~\ref{thm:BDchar} below for further details. 
However, by the characterization
of stochastic completeness for weakly spherically symmetric graphs, e.g., Theorem~9.25 in Section~9.4 of \cite{KLW21}, there exist stochastically incomplete birth-death chains with $m(X)=\infty$.
Thus, the triviality of $\alpha$-harmonic functions in $\ell^1$ does
not imply stochastic completeness and the converse of Theorem~\ref{thm:sc+Feller} does not hold.
\end{example}

\begin{example}[Feller does not imply the triviality of $\alpha$-harmonic functions]\label{ex:solutionButFeller}
We construct a graph $b$ over $(X,m)$ 
allowing for a positive $\alpha$-harmonic function in $\ell^1$ such that the Neumann semigroup is Feller.

The graph is defined by letting $X=\N_0^2$ 
with $m(0,n)=2^{-n}$ and $m(k,n)=1$ for all $k>0$ and $n \geq 0$.
For the edge structure, we set
$$b((k,n),(k+1,n)) =  2^{nk} \qquad \textup{ and } \qquad b((0,n),(0,n+1))=4^{n+2}$$ 
for $n,k \in \N_0$.
Thus, the graph can be visualized as an infinitely high comb where each tooth of the comb
is infinite.

We will first construct a positive $\alpha$-harmonic function which is in $\ell^p(X,m)$ for $p \in [1,\infty]$. 
For the sake of concreteness, we let $\alpha=1$ and consider $u \geq 0$
such that $(\LL + 1)u=0$.
By a Harnack principle, we note that any such $u$ which is non-trivial will be strictly positive.
Thus, we may renormalize to let $u(0,0)=1$. 
Next, for every $n \in \N_0$, we note that
$u(k,n) \to 0$ as $k \to \infty$
since we want $u \in \ell^1(X,m)$ and since $m(k,n)=1$ for $k > 0$.
Now, for $n=0$, from the equations  $(\LL + 1)u(k,0)=0$ for $k=1, 2, 3, \ldots$ and the vanishing
of $u(k,0)$ as $k \to \infty$,
 it can be shown that there
exists a unique value $\beta = (3-\sqrt{5})/2$ such that $u(k,0)=\beta^k$.
The choices $u(0,0)=1$ and $u(1,0)=\beta$ uniquely determine
the value $u(0,1)$ which must be strictly greater than $u(0,0)$. 
In particular, we note that if such a function exists, it is uniquely
determined by the choice of the value $u(0,0)$. 

We will next argue that such a 
function does indeed exist in $\ell^p(X,m)$ for $ p \in [1,\infty]$ but must have infinite energy.
In particular, this will show that $Q^{(D)}=Q^{(N)}$ as form uniqueness
is equivalent to the triviality of $1$-harmonic functions in $D(Q^{(N)}) = \D \cap \ell^2(X,m)$.
To show that $u \in \ell^p(X,m)$, we argue that $u$ will exponentially decay
on the teeth of the comb and, though increasing, must be bounded along the base.

To establish the exponential decay along the teeth, consider the function
$$f(k,n) = u(0,n)2^{-kn}.$$
For $k>0$, a direct calculation gives $\LL f(k,n)=0$. As $u(k,n)$ vanishes as $k \to \infty$,
it is given by the resolvent and thus is the minimal positive solution to $(\LL + 1)w(k,n) \geq 0$ for $k >0$. 
Therefore, $u(k,n) \leq u(0,n)2^{-kn}$ for all $k >0$ and for each $n \in \N_0$.

To show that $u(0,n)$ is bounded for $n \in \N_0$ consider the function
$$g(k,n)=1_{\{k=0\}}\cdot(2-2^{-n}).$$
A direct calculation gives
$$\LL g(0,n)=2^{n}(-4^{n+2}\cdot 2^{-n-1} +4^{n+1}\cdot 2^{-n} + (2-2^{-n})) \leq -(2 - 2^{-n}) = -g(0,n)$$
and thus $(\LL+1)g(0,n) \leq 0$.
Using induction and $(\LL +1)u(0,n)=0$, we get
$$g(0,n+1) -g(0,n) \geq u(0,n+1)-u(0,n)$$
and $g(0,n+1) \geq u(0,n+1)$. In particular, $u(0,n) \leq C$ for all $n \in \N_0$.
By the choice of measure, this shows that there exists $u \in \ell^p(X,m)$ for $p \in [1,\infty]$
uniquely determined by the choice of $u(0,0)$
such that $(\LL+1)u=0$ .

To show that $u$ has infinite energy, we note that the sum of the energies of $u$ along the first tooth 
of each comb is infinite. To see this, observe that $u(0,n)$ is increasing in $n$ and $u(1,n) \leq u(0,n)2^{-n}$
from the exponential decay along the teeth.
Therefore,
$$ \sum_{n=0}^\infty b((0,n), (1,n))[u(0,n) - u(1,n)]^2 \geq
\sum_{n=0}^\infty u^2(0,n)(1 - 2^{-n})^2  = \infty.$$
This establishes form uniqueness since $u \in \ell^2(X,m)$ is the unique, up to scaling, function
such that $(\LL+1)u=0$.

Finally, to show that the Dirichlet semigroup is Feller consider the function
$$h(k,n)=2^{-nk} \cdot 4^{-n}.$$ 
Observe that $\LL h(k,n)=0$ for $k>0$ 
and $\LL h(0,n) \geq 0$ for $n \in \N_0$. As the Dirichlet resolvent generates
the minimal solution, we get for $L=L^{(D)}$ that
 $$(L+1)^{-1} 1_{(0,0)} \leq C h \in C_0(X)$$ 
 for a suitable $C>0$.
This shows that $P_{t}^{(D)}$ is Feller and, thus, that $P_t^{(N)}$
is Feller as we have already established that the semigroups coincide.
Therefore, $P_t^{(N)}$ is Feller but there exists a non-trivial $\alpha$-harmonic
function in $\ell^1(X,m)$.
\end{example}

\section{An edge condition}

A frequently used assumption on graphs is a lower bound on the vertex measure.
This already implies several analytic properties such as the Feller property of the Dirichlet semigroup, see \cite{Woj17}, 
and the coincidence of the Dirichlet and Neumann semigroups, see \cite{HKLW12}.
Recently, a different lower bound has been introduced in order to prove gradient estimates assuming lower Ricci curvature bounds,
see \cite{KM}. We show here that, in the case of locally finite graphs,
this condition implies that the Neumann semigroup is Feller and,
consequently, both form uniqueness and the Feller property of the Dirichlet semigroup hold.

\begin{definition}
Let $(b,c)$ be a graph over $(X,m)$. We say that the graph satisfies
the \emph{edge condition} (EC) if there exists $C>0$ such that for all $x,y \in X$
	\[
	b(x,y) \leq C m(x)m(y).
	\]
\end{definition}

We will show below that (EC) plus local finiteness 
imply that the Neumann semigroup is Feller.
To prove this, we need an $\ell^1$ uniformity of any semigroup associated to the graph coming from a Dirichlet form
as stated in the following lemma. We note that this applies, in particular, to both 
the Neumann and Dirichlet semigroups.

\begin{lemma}\label{l:uniformEll1}
	Let $(b,c)$ be a connected graph over $(X,m)$. Let $T>0$ and $\varphi \in C_c(X)$.
	Let $L$ be an operator associated to the graph such that the associated
form is a Dirichlet form and let $P_t= e^{-tL}$ denote 
the associated semigroup.
	Then, 
	\[
	\sup_{t \in [0,T]} P_t \varphi \in \ell^1(X,m).
	\]
\end{lemma}

\begin{proof}
	Let $x \in X$ and  $t \in [0,T]$.
	We note that the semigroup generates a solution to the heat equation.
	Moreover, the semigroup commutes with the Laplacian and $| P_s \varphi | \leq P_s |\varphi|$ 
	for all $s \geq 0$ 
	as we assume that the associated form is a Dirichlet form and, thus,
	the semigroup is positivity preserving. Using these properties, we calculate
	\begin{align*}
	\sup_{t \in [0,T]}
	|P_t \varphi(x) - \varphi(x)| = 	\sup_{t \in [0,T]}  \left|\int_0^t \LL P_s \varphi(x) ds \right| \leq \int_0^T  P_s |\LL \varphi| (x) ds.
	\end{align*}
	Taking the $\ell^1$ norm and applying Tonelli's theorem gives
	\begin{align*}
	\|\sup_{t \in [0,T]} P_t \varphi - \varphi\|_1 \leq  \int_0^T  \left \|P_s |\LL \varphi| \right\|_1 ds \leq T \|\LL \varphi \|_1 < \infty
	\end{align*}
	due to the compact support of $\varphi$ and the fact that 
	$\LL(C_c(X)) \subseteq \ell^1(X,m)$ which follows
by a direct calculation. Since  $\|\varphi\|_1 < \infty$, this implies
	$\sup_{t \in [0,T]} P_t \varphi \in \ell^1(X,m)$ which finishes the proof.
\end{proof}

We now state and prove the main result of this section. For this, we drop
the killing term and assume local finiteness.

\begin{theorem}\label{thm:ECfeller}
Let $b$ be a locally finite connected graph over $(X,m)$ satisfying the edge condition (EC). 
Then, the Neumann semigroup is Feller.
\end{theorem}

\begin{proof}
	Let $T>0$ and $\varphi \in C_c(X)$ be non-negative.
	We aim to show $P_t \varphi \in C_0(X)$ where $P_t=P_{t}^{(N)}$ is the Neumann semigroup.
	Due to the edge condition (EC), there exists $C>0$ such that $b(x,y) \leq Cm(x)m(y)$ for all $x,y \in X$.	
	Let
	\[
	g= \sup_{t \in [0,T]} P_t \varphi \in \ell^1(X,m)
	\]
	where $g \in \ell^1(X,m)$ due to Lemma~\ref{l:uniformEll1}.
	Let $\eps>0$ and $K \subset X$ be finite such that
	\[
	\|g(1-1_K)\|_1 < \eps.
	\]
	 Without loss of generality, we also assume $\supp \varphi \subseteq K$.

	Now fix $x \in X \setminus \mbox{cl}(K)$ where $\mbox{cl}(K) = K \cup \{ z \mid 
	\mbox{there exists } y \in K \mbox{ such that } b(y,z)>0 \}$. By local
	finiteness, we note that $\mbox{cl}(K)$ is a finite set.
	For $t \in [0,T]$, we estimate
	\begin{align*}
	-\LL P_t \varphi (x) &= \frac 1 {m(x)}\sum_{y \in X \setminus K} b(x,y)(P_t \varphi(y) - P_t \varphi(x))
	\\&\leq  \frac 1 {m(x)}\sum_{y \in X \setminus K} b(x,y)P_t \varphi(y)
	\\&\leq  C \sum_{y \in X \setminus K} m(y) P_t \varphi(y)
	\\&\leq  C \sum_{y \in X \setminus K} m(y) g(y)
	\\& \leq  C \|g(1-1_K)\|_1 < C \eps
	\end{align*}
	where we used (EC) in the second estimate.
	Hence,
	\begin{align*}
	P_T \varphi(x) = P_T \varphi(x) - P_0 \varphi(x) = -\int_0^T \LL P_t \varphi(x) dt \leq TC\eps.
	\end{align*} 
	In particular, $P_T \varphi \leq TC\eps$ on $X \setminus \mbox{cl}(K)$ which proves that $P_t$ is Feller since $\eps$ is arbitrary and since 
	$\mbox{cl}(K)$ is finite. This finishes the proof.
\end{proof}

Combining the result above with Corollary~\ref{cor:Feller_uniqueness}, we obtain the following result.

\begin{corollary}[Edge condition, form uniqueness and the Feller property]\label{cor:ECdirichlet}
Let $b$ be a locally finite connected graph over $(X,m)$ satisfying the edge condition (EC). 
Then, the Dirichlet semigroup is Feller, coincides with the Neumann semigroup
and the restricted Neumann semigroups converge to the Dirichlet
semigroup strongly in $\ell^2(X,m)$ for any exhaustion.
\end{corollary}

\section{Birth-death chains}
We now consider the case of birth-death chains, i.e., graphs with $X=\N_0$, $b(x,y)>0$ if and only if $|x-y|=1$
and $c=0$. In this case we can characterize the Feller property of the Neumann semigroup in terms of the edge weights and vertex measure. We also show that
the Feller property of the semigroup implies the essential self-adjointness of the Laplacian.

We recall that if $\LL(C_c(X)) \subseteq \ell^2(X,m)$, then the operator $L_c = \LL\vert_{C_c(X)}$
is symmetric. It is easy to see that this is always the case for locally finite graphs, in particular,
for birth-death chains.
 If $L_c$ has a unique self-adjoint extension, then we call $L_c$ 
\emph{essentially self-adjoint}. It follows by general abstract principles that essential self-adjointness
always implies form uniqueness, e.g., Corollary~3.7 in Section~3.2 of \cite{KLW21}. We will show below
that a stronger property, that is, the Feller property of the Neumann semigroup implies essential self-adjointness
in the case of birth-death chains.

We start with a lemma establishing some basic properties 
of $\alpha$-harmonic functions on birth-death chains. These will be used
in both results below.

\begin{lemma}\label{lem:bd_harm}
Let $b$ be a birth-death chain over $(\N_0,m)$.
Let $u$ be $\alpha$-harmonic for $\alpha>0$.
If $u$ is non-trivial, then $u(0) \neq 0$. If $u(0)>0$, then $u$
is strictly increasing and 
$$ \sum_{r=0}^\infty u(r)m(r) \geq \widetilde{\alpha} \sum_{r=0}^\infty \frac{m(B_r^c)} {b(r,r+1)}$$
where   $\widetilde{\alpha}=\alpha u(0) m(0)$
and  $B_r^c = \{r+1,r+2, r+3, \ldots\}$.
In particular, if $u \in \ell^1(X,m)$,
then  $\sum_{r=0}^\infty {m(B_r^c)}/{b(r,r+1)} < \infty$.
\end{lemma}
\begin{proof}
If $u(0)=0$, then
$(\LL +\alpha) u  =0$ implies $u=0$ by induction. 

Also by induction, if $u(0)>0$, then $(\LL +\alpha) u  =0$ gives that
$u$ is strictly increasing. 
Furthermore, as now $\LL u \leq 0$, we obtain
$$ b(r,r+1)\large(u(r+1)-u(r)\large) \geq b(r-1,r)\large(u(r)-u(r-1)\large)$$
for all $r \in \N$.
Applying this estimate inductively yields
\begin{align*} 
b(r,r+1)(u(r+1)-u(r)) &\geq b(0,1)(u(1)-u(0))\\ 
&= -m(0)\LL u(0) = \alpha u(0)m(0)=\widetilde{\alpha}.
\end{align*}
Thus, $u(r+1) - u(r) \geq {\widetilde{\alpha}}/{b(r,r+1)}$ yielding
$$ u(r) \geq \widetilde{\alpha} \sum_{k=0}^{r-1} \frac {1}{b(k,k+1)}$$
for $r \geq 1$.
Putting everything together and rearranging gives
$$ \sum_{r=0}^\infty u(r)m(r) \geq \widetilde{\alpha} \sum_{r=1}^\infty \sum_{k=0}^{r-1} \frac{m(r)}{b(k,k+1)} = 
\widetilde{\alpha} \sum_{r=0}^\infty \frac{\sum_{k= r+1}^\infty m(k)} {b(r,r+1)} = \widetilde{\alpha} \sum_{r=0}^\infty \frac{m(B_r^c)} {b(r,r+1)}.$$
The ``in particular'' statement is then obvious. 
\end{proof}

We now characterize the Feller property of the Neumann semigroup on birth-death chains via
$\alpha$-harmonic functions in $\ell^1$ as well as conditions on the edge weights and vertex measure.
This combines our previous results, including the lemma directly above, 
with some known characterizations for form uniqueness and the Feller property for the Dirichlet semigroup
and complements known characterizations on birth-death chains for stochastic completeness
and recurrence.

\begin{theorem}\label{thm:BDchar}
	Let $b$ be a birth-death chain over $(\N_0,m)$. The following statements are equivalent:
	\begin{itemize}
		\item[\textup{(i)}]  The Neumann semigroup is Feller.
		\item[\textup{(ii)}]  Every $\alpha$-harmonic function in $\ell^1(X,m)$ for $\alpha>0$ is trivial.
		\begin{itemize}
		\item[\textup{(ii$'$)}] Every positive $\alpha$-harmonic function in $\ell^1(X,m)$ for $\alpha>0$ is trivial.
		\end{itemize}
		\item[\textup{(iii)}]  We have $m(X) = \infty$ or 
		\[
		\sum_r \frac 1{b(r,r+1)} = \infty =  \sum_r \frac{m(B_r^c)} {b(r,r+1)}
		\]
		where $B_r^c = \{r+1,r+2, r+3, \ldots\}$.
	\end{itemize}
\end{theorem}

\begin{proof}
That (i) $\Longleftrightarrow$ (iii) follows from the following characterizations of the Feller property for the Dirichlet
semigroup and form uniqueness on birth-death chains. First, by Theorem~4.13 in \cite{Woj17}, the Dirichlet
semigroup is Feller if and only if 
$$\sum_{r} \frac{1}{b(r,r+1)} < \infty \qquad \textup{ or } \qquad
\sum_r \frac{m(B_r^c)} {b(r,r+1)}=\infty.$$
On the other hand, form uniqueness holds if and only if 
$$\sum_{r} \frac{1}{b(r,r+1)} = \infty \qquad \textup{ or } \qquad  m(X)=\infty.$$ 
This can be seen follows:
If $m(X)=\infty$, then $L_c$ is essentially self-adjoint by Theorem~6 in \cite{KL12} and thus form uniqueness holds by general principles, e.g.,
Corollary~3.7 in Section~3.2 of \cite{KLW21}.
If $m(X)<\infty$, then form uniqueness is equivalent to recurrence by Theorem~16 in \cite{Sch17} and it is well-known that
recurrence on birth-death chains is equivalent to $\sum_r 1/b(r,r+1) =\infty$, 
e.g., Theorem~5.9 in \cite{Woe09}  or Theorem~9.21 in Section~9.3 of \cite{KLW21}.
Now, the equivalence of (i) and (iii) follows by combining the 
above characterizations and using Corollary~\ref{cor:Feller_uniqueness}.

That (ii) $\Longrightarrow$ (ii$'$) is obvious
and that (ii$'$) $\Longrightarrow$ (i) 
follows from Theorem~\ref{thm:harmonic_Feller}
as birth-death chains are locally finite.

We now prove (iii) $\Longrightarrow$ (ii). We argue by contraposition so
suppose there exists a non-trivial $u \in \ell^1(X,m)$ with $(\LL +\alpha) u  =0$ for $\alpha>0$. By Lemma~\ref{lem:bd_harm}, we know that $u(0) \neq 0$
and by rescaling we may assume $u(0)>0$. 
Thus, as $u$ is then strictly increasing
we obtain $m(X)<\infty$ from $u \in \ell^1(X,m)$ and $\sum_r {m(B_r^c)}/{b(r,r+1)}<\infty$
from Lemma~\ref{lem:bd_harm}. This completes the proof of the theorem. 
\end{proof}

As a consequence, we now show that the Feller property of the Neumann semigroup implies essential self-adjointness on birth-death chains.
We recall that
$L_c$ is essentially self-adjoint if and only if every $\alpha$-harmonic
function in $\ell^2(X,m)$ for $\alpha>0$ is trivial, e.g., Theorem~3.6 in Section~3.2 of \cite{KLW21}.
Using this and the results above yields the following
statement.

\begin{corollary}
	Let $b$ be a birth-death chain over $(X,m)$. 
	If the Neumann semigroup is Feller, then $L_c$ is essentially self-adjoint.
\end{corollary}
\begin{proof}
If $L_c$ is not essentially self-adjoint, then there exists 
a non-trivial $u \in \ell^2(X,m)$ with $(\LL+\alpha) u = 0$ for $\alpha>0$.
By Lemma~\ref{lem:bd_harm} and rescaling we assume, 
without loss of generality, that $u(0)=1$. As $u$ must then be strictly increasing, 
this implies $m(X)<\infty$. Furthermore, as $\|u\|_2 \geq \|u\|_1$,
Lemma~\ref{lem:bd_harm} implies 
$\sum_{r=0}^\infty {m(B_r^c)}/{b(r,r+1)}<\infty$ and
thus the Neumann semigroup is not Feller by Theorem~\ref{thm:BDchar}.
\end{proof}

\begin{rem}
The corollary above can also be obtained from Theorem~\ref{thm:BDchar} 
and the following characterization
of essential self-adjointness for the Laplacian on birth-death chains due to Hamburger 
\cite{Ham20a, Ham20b}, see also \cite{IMW}. 
Namely, on birth-death chains, the essential self-adjointness
of $L_c$ is equivalent to
$$\sum_{r=0}^\infty \left( \sum_{k=0}^r \frac{1}{b(k,k+1)} \right)^2 m(r+1)
=\infty.$$
This also shows that the converse to the above corollary does not hold.
\end{rem}

\subsection*{Acknowledgments}
R.K.W.~would like to thank Isaac Chavel, J{\'o}zef Dodziuk and Leon Karp for a discussion that 
led to his interest in the Neumann exhaustion. Furthermore, he would like to thank Daniel Lenz,
Marcel Schmidt and Hendrik Vogt for helpful exchanges and the organizers of the ISEM 26 workshop
where this work was first presented. Finally, he acknowledges financial support from the Simons Foundation
in the form of a Travel Support for Mathematicians gift. 
R.K.W.~and~F.M.~are grateful to Christian Rose for helpful discussions on the topics of this paper.
M.K.~acknowledges financial support of the DFG.
\\

\printbibliography

Matthias Keller,\\
Department of Mathematics,
University of Potsdam, Potsdam, Germany\\
\texttt{mkeller@math.uni-potsdam.de}\\

Florentin M{\"u}nch, \\
Max Planck Institute for Mathematics in the Sciences Leipzig, Germany \\
\texttt{cfmuench@gmail.com}\\

Rados{\l}aw K. Wojciechowski,\\
York College and the Graduate Center of the City University of New
York, New York, USA\\
\texttt{rwojciechowski@gc.cuny.edu}

\end{document}